\documentclass[12pt]{amsart}
\usepackage{amsmath}
\usepackage{amsfonts}
\usepackage{amssymb}
\usepackage{amsthm}
\usepackage{eucal}
\newtheorem{theorem}{Theorem}[section]

\usepackage{mathrsfs}
\usepackage{eucal}
\usepackage{graphicx}
\newcommand{\qq}{\mathbb{Q}}
\newcommand{\zz}{\mathbb{Z}}
\newcommand{\cc}{\mathbb{C}}
\newcommand{\pp}{\mathbb{P}}
\renewcommand{\gg}{\mathbb{G}}

\newcommand{\E}{\mathcal{E}}

\renewcommand{\O}{\mathscr{O}}

\renewcommand{\tilde}{\widetilde}
\newcommand{\B}{\mathcal{B}}
\newcommand{\U}{\mathcal{U}}
\DeclareMathOperator{\codim}{codim}
\DeclareMathOperator{\Aut}{Aut}
\newtheorem{Lemma}[theorem]{Lemma}
\newtheorem{Corollary}[theorem]{Corollary}

\DeclareMathOperator{\rank}{rank}
\theoremstyle{definition}
\newtheorem{Example}[theorem]{Example}
\theoremstyle{remark}
\newtheorem*{remark}{Remark}
\theoremstyle{remark}

\usepackage{tikz}
\usepackage{tikz-cd}
\usepackage{wasysym, stackengine, makebox, graphicx}

\newcommand{\V}{\mathcal{V}}
\newcommand{\W}{\mathcal{W}}
\newcommand{\GL}{\mathrm{GL}}
\newcommand{\PGL}{\mathrm{PGL}}
\newcommand{\BGL}{\mathrm{BGL}}
\newcommand{\SL}{\mathrm{SL}}
\newcommand{\ch}{\mathrm{ch}}
\newcommand{\td}{\mathrm{td}}
\newcommand{\CH}{\mathrm{CH}}

\tikzset{cong/.style={draw=none,edge node={node [sloped, allow upside down, auto=false]{$\cong$}}},
         Isom/.style={draw=none,every to/.append style={edge node={node [sloped, allow upside down, auto=false]{$\cong$}}}}}

\usetikzlibrary{matrix,arrows,decorations.pathmorphing}
\usepackage[margin=1in]{geometry}

\numberwithin{equation}{section}

\title{The intersection theory of the moduli stack of vector bundles on $\pp^1$}
\author{Hannah K. Larson}
\thanks{During the preparation of this article, the author was supported by the Hertz Foundation and NSF GRFP under grant DGE-1656518.}
\begin{document}
\maketitle

\begin{abstract}
We determine the integral Chow and cohomology rings of the moduli stack $\B_{r,d}$ of rank $r$, degree $d$ vector bundles on $\pp^1$ bundles.
We first show that the rational Chow ring $A_\qq^*(\B_{r,d})$ is a free $\qq$-algebra on $2r+1$ generators. The isomorphism class of this ring happens to be independent of $d$. Then, we prove that the integral Chow ring $A^*(\B_{r,d})$ is torsion-free and provide multiplicative generators for $A^*(\B_{r,d})$ as a subring of $A_{\qq}^*(\B_{r,d})$. From this description, we see that $A^*(\B_{r,d})$ is not finitely generated as a $\zz$-algebra.
Finally, the cohomology ring of $\B_{r,d}$ is isomorphic to its Chow ring.
\end{abstract}

\section{Introduction}

In this paper, we fully describe the intersection theory of the moduli stack $\B_{r,d}$ of vector bundles on $\pp^1$ bundles. Precisely, an object of $\B_{r,d}$ over a scheme $T$ is the data of a rank $2$ vector bundle $W$ on $T$ and a rank $r$, relative degree $d$ vector bundle $E$ on $\pp W$.
To describe generators of the Chow or cohomology ring,
let $\pi: \pp \W \rightarrow \mathcal{B}_{r,d}$ be the universal $\pp^1$ bundle and let $w_1 = c_1(\W)$ and $w_2 = c_2(\W)$ be Chern classes of the universal rank $2$ bundle $\W$.
Let $\E$ be the universal rank $r$ bundle on $\pp \W$. If $z = c_1(\O_{\pp \W}(1))$, the Chern classes of $\E$ are uniquely expressible as $c_i(\E) = \pi^*(a_i) + \pi^*(a_i') z$ for $a_i \in A^i(\B_{r,d})$ and $a_i' \in A^{i-1}(\B_{r,d})$. We show that the rational Chow ring of $\B_{r,d}$ is freely generated by these classes. Then, we show that the integral Chow ring is torsion-free, and describe it as a subring of the rational Chow ring. This also determines the cohomology ring of $\B_{r,d}$, as it agrees with the Chow ring.

\begin{theorem}  \label{main}
We have $A_{\qq}^*(\B_{r,d}) = \qq[w_1, w_2, a_1, \ldots, a_r, a_2', \ldots, a_r']$. The integral Chow ring $A^*(\B_{r,d}) \subset A_{\qq}^*(\B_{r,d})$ is the subring generated by $w_1, w_2$ and the Chern classes of $\pi_*\E(i)$ for $i = 0, 1, 2, \ldots$. Moreover, the cycle class map $A^*(\B_{r,d}) \to H^{2*}(\B_{r,d})$ is an isomorphism.
\end{theorem}

In the special case when the $\pp^1$ bundle is trivial ($W$ is trivial), the integral Chow ring has a somewhat simpler description. This describes the Chow ring of the moduli space $\B_{r,d}^\dagger$ of vector bundles on a fixed $\pp^1$. (Precisely, an object of $\B_{r,d}^\dagger$ over a scheme $T$ is a rank $r$, relative degree $d$ vector bundle $E$ on the trivial $\pp^1$ bundle $\pp^1 \times T$.) The map $\B_{r,d}^\dagger \to \B_{r,d}$ is the $\GL_2$ bundle associated to $\W$ on $\B_{r,d}$. By \cite[Theorem 2]{V} of Vistoli, the pullback $A^*(\B_{r,d}) \to A^*(\B_{r,d}^\dagger)$ is surjective with kernel generated by $w_1, w_2$.
 
\begin{theorem} \label{m2}
We have $A_{\qq}^*(\B_{r,d}^\dagger) = \qq[a_1, \ldots, a_r, a_2', \ldots, a_r']$.
Integrally, $A^*(\B_{r,d}^\dagger) \subset A_{\qq}^*(\B_{r,d}^\dagger)$ is the subring generated by
$a_1, \ldots, a_r$ and the coefficients $f_i$ of the power series
\begin{equation*}
\sum_{i=0}^\infty f_i t^i = \exp\left(\int \frac{d \cdot (a_1 +a_2t +\ldots + a_r t^{r-1}) - (a_2' + a_3't + \ldots + a_r' t^{r-2})}{1 + a_1t + \ldots + a_rt^r}dt \right).
\end{equation*}
Moreover, the cycle class map $A^*(\B_{r,d}^\dagger) \to H^{2*}(\B_{r,d}^\dagger)$ is an isomorphism.
\end{theorem}

We point out several interesting features of our results:
\begin{enumerate}
\item Although $A^*_{\qq}(\B_{r,d})$ is obviously finitely generated as a $\qq$ algebra, $A^*(\B_{r,d})$ is \emph{not} finitely generated as a $\zz$ algebra (see Corollary \ref{nfg}).
\item The rational Chow ring $A^*_{\qq}(\B_{r,d})$ is independent of $d$. 
This may lead one to wonder if the isomorphism class of $\B_{r,d}$ could be independent of $d$. However, considering integral Chow rings, one can show that $\B_{2,1}$ and $\B_{2,0}$ (resp. $\B_{2,1}^\dagger$ and $\B_{2,0}^\dagger$) are \emph{not} isomorphic (see Corollary \ref{b2}).
\item To show $A^*(\B_{r,d})$ is torsion-free, we stratify by splitting loci, which in turn are modeled by spaces admitting affine stratifications. This stratification is also what allows us to see that the Chow and cohomology rings of $\B_{r,d}$ agree (see Lemma \ref{coh}). 
\item Using the theory of higher Chow groups, we show that push forward maps for including strata are all injective on Chow. This relies on  a vanishing result for higher Chow groups of a point with torsion coefficients. Although this vanishing result only holds for an algebraically closed field of characteristic zero, we deduce from it a rank equality which allows us to establish our theorem in \emph{all} characteristics.
\end{enumerate}

\begin{remark}
Here, we are considering $\pp^1$ bundles equipped with a relative degree $1$ line bundle. However, not all families of genus $0$ curves admit a relative degree $1$ line bundle. Nevertheless, our work does determine the rational Chow ring $A_{\qq}^*([\B_{r,d}^\dagger/\PGL_2])$, because it is equal to $A_{\qq}^*([\B_{r,d}^\dagger/\SL_2])$ (the $\SL_2$ quotient is a $\mu_2$ gerbe over the $\PGL_2$ quotient). To find the latter, note that $[\B_{r,d}^\dagger/\SL_2] \to \B_{r,d}$ is the $\gg_m$ bundle associated to $\det \W$, so applying Vistoli's theorem  \cite[Theorem 2]{V}, we have  
\[A_{\qq}^*([\B_{r,d}^\dagger/\SL_2]) = A_{\qq}^*(\B_{r,d})/\langle w_1 \rangle = \qq[w_2, a_1, \ldots, a_r, a_2', \ldots, a_r'].\]
 This result will be used to determine the intersection theory of low-degree Hurwitz spaces by S. Canning and the author in \cite{CL}.
\end{remark}

This paper is organized as follows. In Section \ref{2}, we briefly state necessary results concerning equivariant Chow rings and higher Chow groups. In Section \ref{con}, we describe a sequence of opens $\U_m$ that exhaust $\B_{r,d}$. Following a construction of Bolognesi--Vistoli, each $\U_m$ can be realized as a global quotient.
In Section \ref{chow}, we use this description to calculate the rational Chow ring of $\B_{r,d}$. Finally, in Section 5, we prove that the integral Chow ring is torsion-free and provide a description of the generators.

\subsection*{Acknowledgements} I would like to thank Ravi Vakil for many helpful conversations and in particular pointing me towards ideas of Akhil Mathew and Eric Larson on higher Chow groups. I am grateful to the latter two for explaining higher Chow groups and how to use them. Thanks also to Samir Canning for thoughtful conversations about this work and comments on an earlier version of this paper. 

\section{Preliminaries on equivariant Chow and higher Chow} \label{2}
Most of the stacks we encounter in this paper are quotients of open subsets $X$ of affine space by linear algebraic groups $G$. The Chow ring of a quotient stack $[X/G]$ is defined as the $G$-equivariant Chow ring of $X$, which in turn is defined in \cite{EG} using models based on Borel's mixing construction. More precisely, given a representation $V$ of $G$ and an open subset $U \subset V$ on which $G$ acts freely, if $\codim U^c > i$, we define $A^i([X/G]) = A^i(X \times_G U)$ where $X \times_G U$ is the quotient of $X \times U$ by the diagonal $G$ action.
This is well-defined because $A^*(V) \cong A^*(X)$ whenever $V$ is a vector bundle over $X$ (``homotopy") and $A^i(X - Z) = A^i(X)$ whenever $\codim Z > i$ (``excision").

\begin{Example} \label{bglr}
Let $V_N = \mathrm{Mat}_{r \times N}(k) = k^{\oplus rN}$ with $\GL_r$ acting on $V_N$ by left multiplication.
The group $\GL_r$ acts freely on the open subset $U_N \subset V_N$ of full rank matrices. This determines a model $pt \times_{\GL_r} U_N = U_N/\GL_r = G(r, N)$. Since $\codim U_N^c = N - r + 1$, we have $A^i(\BGL_r) = A^i(G(r,N))$ for $i < N - r + 1$. It is a classical result that $A^*(G(r,N))$ is generated by the Chern classes $c_1, \ldots, c_r$ of the tautological rank $r$ bundle with no relations in degrees less than $N - r +1$. Taking larger and larger $N$, it follows that $A^*(\BGL_r) = \zz[c_1, \ldots, c_r]$.
\end{Example}

Variants of the above construction allow one to approximate all quotient stacks in this paper with concrete models which are fiber bundles over Grassmannians.

The higher Chow groups of these quotient stacks will also be an important tool.
In \cite{B}, Bloch defines the higher Chow groups of a quasi-projective variety $X$ as the homology of a complex $z^*(X,-)$ of free abelian groups, i.e. $\CH^*(X, n) = H_n(z^*(X,-))$.
Higher Chow groups with coefficients in a ring $R = \zz/m$ or $\zz$ are defined similarly by $\CH^*(X, n, R) = H_n(z^*(X,-) \otimes R)$.
Some properties of higher Chow groups are the following.
 \begin{enumerate}
 \item Weight zero: we have $\CH^*(X, 0, R) = A^*(X) \otimes R$.
 \item \label{func} Functoriality: there are proper push forwards and and flat pull backs.
 \item \label{loc} Localization long exact sequence: If $Y \subset X$ is a closed subscheme of pure codimension $d$, then there is a long exact sequence
 \begin{align*}
 \ldots &\rightarrow \CH^{*-d}(Y, 1, R) \rightarrow \CH^*(X,1,R) \rightarrow \CH^*(X-Y,1,R) \\
 &\rightarrow \CH^{*-d}(Y,0,R) \rightarrow \CH^*(X,0,R) \rightarrow \CH^*(X-Y,0,R).
 \end{align*}
 \item Homotopy: $\CH^*(X \times \mathbb{A}^m, n, R) \cong \CH^*(X, n, R)$. By \eqref{func} and \eqref{loc} it follows that if $\widetilde{X} \rightarrow X$ is any affine bundle, then $\CH^*(X, n, R) \cong \CH^*(\tilde{X},n, R)$.  \label{htopy}
 \end{enumerate}
 
Edidin and Graham \cite{EG} extend the notion of higher Chow groups to quotients $[X/G]$ by defining them to be higher Chow groups of suitable models: $\CH^*([X/G], n, R) := \CH^*(X \times_G U, n, R)$ where $U$ is an open subset of a representation of $G$ whose complement has sufficiently high codimension and $G$ acts freely on $U$. This is well-defined by the homotopy property, and Edidin-Graham obtain a localization long exact sequence for the corresponding quotients in \eqref{loc} when $Y$ is $G$-equivariant.

Over an algebraically closed field of characteristic zero, the higher Chow groups of a point with torsion coefficients are known:
\begin{equation}
\CH^i(pt, n, \zz/\ell) = \begin{cases} \zz/\ell &\text{if $n = 2i$} \\ 0 & \text{otherwise.} \end{cases}
\end{equation}
This follows from \cite[Corollary 4.3]{Suslin}, which relates higher Chow groups to certain \'etale cohomology groups, though this special case was likely known earlier.
Using the long exact sequence and the homotopy property, it follows that over such a field, $\CH^*(X, 1; \zz/\ell) = 0$ for any $X$ admitting an affine stratification.
In particular, since $\BGL_{r}$ is modeled by Grassmannians $G(r, N)$ (see Example \ref{bglr}), we have
\begin{equation} \label{hch}
\CH^*(\BGL_{r_1} \times \cdots \times \BGL_{r_s}, 1; \zz/\ell) = 0
\end{equation}
over an algebraically closed field of characteristic zero.

\section{Construction of the moduli stack} \label{con}

Given some rank $r \geq 0$ and degree $d \in \zz$, we define the moduli stack $\B_{r,d}$ of vector bundles of rank $r$ and degree $d$ on $\pp^1$ bundles by
\[\B_{r,d}(T) = \left\{(W,E): \parbox{22em}{$W$ a rank $2$ vector bundle on $T$ \\
$E$ a rank $r$, relative degree $d$ vector bundle on $\pp W$}\right\}.\]
An arrow $(W, E) \to (W', E')$ is the data of an isomorphism $\phi: W \to W'$ --- which induces an isomorphism $\pp \phi: \pp W \to \pp W'$ --- and an isomorphism $\psi: E \to (\pp \phi)^* E'$.
Given a vector bundle $E$ on a $\pp^1$ bundle $\pp W \rightarrow T$, we write $E(m) := E \otimes \O_{\pp W}(1)^{\otimes m}$.
There are equivalences $\B_{r, d} \cong \B_{r,d+mr}$ sending $(W, E) \mapsto (W, E(m))$. Thus, it would suffice to study $\B_{r,\ell}$ for $0 \leq \ell < r$. Throughout, $d = \ell + mr$ will be an arbitrary degree in the residue class of $\ell$ modulo $r$.

The stack $\B_{r,\ell}$ is a union of open substacks
\[\B_{r,\ell} = \bigcup_{m=0}^\infty \U_{m} \qquad \text{with} \qquad \U_0 \subset \U_1 \subset \U_2 \subset \cdots.\]
where
$\U_m := \U_{m, r, \ell}$ is defined by
\begin{align*}
\U_{m, r, \ell}(T) &= \{(W, E) \in \B_{r,\ell}(T) : \text{$E(m)$ is globally generated on each fiber over $T$}\} \\
&=\{(W, E) \in \B_{r,\ell}(T) : \text{$R^1\pi_*E(m-1) = 0$ for $\pi: \pp W \rightarrow T$ projection}\}.
\end{align*}
The stack $\U_{0, r, \ell}$ is equal to the stack of globally generated rank $r$, degree $\ell$ vector bundles on $\pp^1$ bundles, which we shall denote $\V_{r,\ell}$. Via the twist by $\O_{\pp W}(m)$, there are isomorphisms $\U_{m, r, \ell} \cong \V_{r, \ell + mr}$ for each $m$. 

\subsection{Splitting loci} \label{sps}
Every vector bundle on $\pp^1$ splits as a direct sum of line bundles, say $E = \O(e_1) \oplus \cdots \oplus \O(e_r)$ for $e_1 \leq \cdots \leq e_r$. We call $\vec{e} = (e_1, \ldots, e_r)$ the \emph{splitting type} of $E$, and abbreviate the corresponding sum of line bundles by
\[\O(\vec{e}) := \O(e_1) \oplus \cdots \oplus \O(e_r).\]
The moduli space $\B_{r,\ell}$ admits a stratification by the \emph{splitting loci} of the universal vector bundle.
These are the locally closed substacks $\Sigma_{\vec{e}} \subset \B_{r,\ell}$ which parametrize families of vector bundles with splitting type $\vec{e}$ on each fiber of a $\pp^1$ bundle.

Suppose that $\O(\vec{e}) = \bigoplus_{i=1}^s \O(d_i)^{\oplus r_i}$ with $d_1 < \ldots < d_s$ (so the $d_i$ are the \emph{distinct} degrees in $\vec{e}$ and the $r_i$ their multiplicities).
Let us identify $\Aut(\O(\vec{e}))$ with block upper triangular matrices whose entries in the $i, j$ block are homogeneous polynomials on $\pp^1$ of degree $d_j - d_i$. 
The block diagonal matrices correspond to the subgroup $\prod_{i=1}^s \GL_{r_i} \hookrightarrow \Aut(\O(\vec{e}))$.
The group $\GL_2$ acts on the off-diagonal blocks via change of coordinates on $\pp^1$. 
The data of a vector bundle $E$ on a $\pp^1$ bundle $\pp W \rightarrow T$ whose restriction to each fiber has splitting type $\vec{e}$
is the same as a principal bundle for the semi-direct product $H_{\vec{e}} := \Aut(\O(\vec{e})) \ltimes \GL_2$.
In other words, $\Sigma_{\vec{e}}$ is equivalent to the classifying stack $B H_{\vec{e}}$. From this, we see that
\begin{equation} \label{codimeq}
\codim \Sigma_{\vec{e}} = h^1(\pp^1, End(\O(\vec{e})) = \sum_{i, j} \max\{0, e_i - e_j-1\} =: u(\vec{e}).
\end{equation}

The complement of $\U_{m,r,\ell} \subset \B_{r,\ell}$ 
is the union of splitting loci $\Sigma_{\vec{e}}$ with $e_1 < -m$. In particular, using \eqref{codimeq}, one sees that this union of splitting loci has codimension $\ell + mr + 1$ (see also \cite[Section 5]{L}). In particular, as $m$ increases, the codimension of the complement of $\U_m$ goes to infinity.
The key to understanding the intersection theory of $\B_{r, \ell}$ is therefore to understand each $\U_{m,r,\ell}$, which is equivalent to the stack of globally generated vector bundles $\V_{r, d}$ for $d = \ell + mr$.

\subsection{Globally generated vector bundles}
The stack $\V_{r,d}$ of rank $r$, degree $d$ globally generated vector bundles on $\pp^1$ was constructed by Bolognesi--Vistoli in \cite{BV}. We briefly motivate and review their construction.
 If $E$ is a globally generated vector bundle on $\pp^1$, then there is a surjection $H^0(E) \otimes \O_{\pp^1} \to E$. Since $h^0(E) = r + d$, the kernel of this map is rank $d$, degree $-d$. Furthermore, the kernel has no global sections, so it must be $\O_{\pp^1}(-1)^{\oplus d}$. That is, given a globally generated $E$, it sits naturally in a sequence
\begin{equation} \label{ptss}
\begin{tikzcd}
0 \arrow{r} & \O_{\pp^1}(-1)^{\oplus d} \arrow{r}{\psi} & \O_{\pp^1}^{\oplus (r + d)} \arrow{r} & E \arrow{r} &0.
\end{tikzcd}
\end{equation}

Let $M_{r,d} := \mathrm{Hom}( \O_{\pp^1}(-1)^{\oplus d}, \O_{\pp^1}^{\oplus (r + d)}) $ be the space of $d \times (r+d)$ matrices of linear forms on $\pp^1$.
The sequence \eqref{ptss} determines an element $\psi \in M_{r,d}$ which is well-defined up to the choice of framings of the source and target. Moreover, $\psi$ lies in the open subvariety $\Omega_{r,d} \subset M_{r,d}$ of matrices of linear forms having full rank $d$ at each point on $\pp^1$. 
Let $\GL_d$ act on $M_{r,d}$ by left multiplication and $\GL_{r+d}$ act by right multiplication. In addition, let $\GL_2$ act on $M_{r,d}$ by change of coordinates on the entries, which live in the two-dimensional vector space $H^0(\pp^1, \O_{\pp^1}(1))$. These three actions commute, so we obtain an action of $\GL_d \times \GL_{d + r} \times \GL_2$ on $M_{r, d}$. The locus $\Omega_{r,d} \subset M_{r,d}$ is preserved by this action and thus inherits an action of $\GL_d \times \GL_{d + r} \times \GL_2$.

\begin{theorem}[Bolognesi-Vistoli \cite{BV}, Theorem 4.4] \label{bvthm}
There is an isomorphism of fibered categories
$\V_{r,d}\cong[\Omega_{r,d}/\GL_d \times \GL_{r+d} \times \GL_2].$
\end{theorem}

In other words, Theorem \ref{bvthm} says $\V_{r,d}$ is an open substack of a vector bundle over $\BGL_d \times \BGL_{r+d} \times \BGL_2$. 
In particular, by the homotopy and excision properties, we have a surjection 
\begin{equation} \label{surj}
A^*(\BGL_d \times \BGL_{r+d} \times \BGL_2) \twoheadrightarrow A^*(\V_{r,d}).
\end{equation}
Let $\W$ denote the universal rank $2$ vector bundle on $\V_{r,d}$ (pulled back from the $\BGL_2$ factor), and let $\E^{gg} := \E_{r,d}^{gg}$ be the universal globally generated rank $r$, degree $d$ vector bundle on $\pi: \pp \W \to \V_{r,d}$.
Let $T_d$ and $T_{r+d}$ denote the universal vector bundles on $\BGL_d$ and $\BGL_{r+d}$. We wish to identify their pullbacks to $\V_{r,d}$ in terms of $\E^{gg}$.

\begin{Lemma} \label{genslem}
Let $\gamma: \V_{r,d} \rightarrow \BGL_d \times \BGL_{r+d}$ be the natural map. 
We have
\begin{equation} \label{id}
\gamma^* T_d = (\pi_* \E^{gg}(-1)) \otimes \det \W^\vee \qquad \text{and} \qquad \gamma^*T_{r+d} = \pi_* \E^{gg}.
\end{equation}
In particular, $A^*(\V_{r,d})$ is generated by the Chern classes of the three vector bundles $\W$, $\pi_*\E^{gg}(-1)$ and $\pi_* \E^{gg}$.
\end{Lemma}
\begin{proof}
By the construction of $\V_{r,d}$ as a quotient of $\Omega_{r,d} \subset M_{r,d}$, the universal $\pp^1$-bundle $\pi: \pp \W \to \V_{r,d}$ is equipped with an exact sequence of vector bundles
\begin{equation} \label{eseq}
0 \rightarrow (\pi^*\gamma^*T_d)(-1) \rightarrow \pi^*\gamma^*T_{r+d} \rightarrow \E^{gg} \rightarrow 0.
\end{equation}
Pushing forward \eqref{eseq} by $\pi$ induces an isomorphism 
\[\gamma^*T_{r+d} \cong \pi_*\pi^* \gamma^* T_{r+d} \xrightarrow{\sim} \pi_* \E^{gg}.\]
On the other hand, tensoring \eqref{eseq} with $\O_{\pp \W}(-1)$ and pushing forward by $\pi$ induces an isomorphism
\begin{equation} \label{next}
\pi_*\E^{gg}(-1) \xrightarrow{\sim} R^1\pi_*( (\pi^* \gamma^* T_d)(-2)) \cong \gamma^* T_d \otimes R^1 \pi_*  \O_{\pp \W}(-2) 
\end{equation}
Noting that the relative dualizing sheaf of $\pi$ is $\omega_{\pi} = \O_{\pp \W}(-2) \otimes \det \W^\vee$, Serre duality provides an isomorphism of the righthand term in \eqref{next} with $\pi^*\gamma^*T_d \otimes \det \W$. Having identified the tautological vector bundles, \eqref{surj} now establishes the claim about generators of $A^*(\V_{r,d})$.
\end{proof}

\begin{remark}
Lemma \ref{genslem} provides a quick proof of the existence half of \cite[Theorem 1.2]{L}: Pulling back the classes of closures of the universal splitting loci $\overline{\Sigma}_{\vec{e}}$ on $\B_{r,d}$, it follows that
when the splitting loci of a vector bundle $E$ on a $\pp^1$ bundle $\pp W \rightarrow B$ have the expected codimension, their classes in the Chow ring of $B$ are given by a universal formula in terms of the Chern classes of the rank $2$ bundle $\pi_*\O_{\pp W}(1)$ and the bundles $\pi_*E(i)$ for suitable $i$. 
This observation does not, however, give an indication of how to find these formulas, as done in \cite[Section 6]{L}.
\end{remark}

Now, let $\E := \E_{r,\ell}$ denote the universal rank $r$, degree $\ell$ vector bundle on $\pp \W \to \B_{r,\ell}$. The restriction of $\E$ to $\U_{0, r,\ell} \subset \B_{r,\ell}$ is  just $\E|_{\U_{0,r,\ell}} = \E^{gg}_{r,\ell}$.  More generally, we have each $\U_{m, r,\ell} \cong \V_{r,\ell+mr}$, and via this identification, $\E|_{\U_{m, r, \ell}} = \E_{r, \ell+mr}^{gg}(-m)$, equivalently $\E(m)|_{\U_{m, r, \ell}} = \E^{gg}_{r,\ell+mr}$. This establishes the following.

\begin{Lemma} \label{int-gen}
The Chow ring $A^*(\U_m)$ is generated over $\zz$ by the Chern classes of $\W$ and the vector bundles $\pi_*\E(m-1)$ and $\pi_*\E(m)$ on $\U_m$. Thus, $A^*(\B_{r,\ell})$ is generated by the Chern classes of $\W$ and $\pi_*\E(i)$ for $i = 0, 1, 2, \ldots$
\end{Lemma}

 We shall later describe the Chow ring as a subring of a finitely generated $\qq$-algebra, which gives rise to an implicit description of the relations among these generators.

\section{The rational Chow ring} \label{chow}
The rational Chow ring of $\B_{r,\ell}$ can be described with fewer generators than the integral generators of Lemma \ref{int-gen}.
Let $w_1 = c_1(\W)$, $w_2 = c_2(\W)$ and $z = c_1(\O_{\pp \W}(1))$.
 The Chern classes of $\mathcal{E}$ can be written as $c_i(\mathcal{E}) = \pi^*(a_i) + \pi^*(a_i')z$ for unique $a_i \in A^i(\B_{r,\ell})$ and $a_i' \in A^{i-1}(\B_{r,\ell})$. (Note that $a_1' = \ell$.) By Grothendieck-Riemann-Roch, the Chern classes of the vector bundle $\pi_*\E(m)$ on $\U_m$ are expressible in terms of the $a_i$ and $a_i'$ for any $m$, so these classes are generators for the rational Chow ring of each $\U_m$ and therefore of $\B_{r,\ell}$. The main result of this section will be that there are no relations among these generators on $\B_{r,\ell}$. We first consider relations on an open $\U_m \cong \V_{r,d}$.
 
\begin{theorem} \label{ratchow}
The ring $A_{\qq}^*(\V_{r,d})$ is a quotient of $\qq[w_1, w_2, a_1, \ldots, a_r, a_2', \ldots, a_r']$ with all relations in degrees $d+1$ and higher. 
\end{theorem}

\begin{proof}
Set $G := \GL_d \times \GL_{r+d} \times  \GL_2$.
We let $T_d, T_{r+d}$ and $\W$ denote the corresponding tautological bundles on $B := BG$. Let $t_i = c_i(T_{d})$ and $u_i = c_1(T_{r+d})$ and $w_i = c_i(\W)$ be their Chern classes, which freely generate $A^*(B)$.
Next, define $M:=[M_{r,d}/G]$, which is the total space of the vector bundle $\mathcal{H}om(T_d,T_{r+d}) \otimes \W^\vee$ over $B$. Theorem \ref{bvthm} says that $\V_{r,d}$ is an open substack of $M$. Let us write $X:=[\Omega_{r,d}^c/G]$ for its closed complement. Here, $\Omega_{r,d}^c \subset M_{r,d}$ is the space of matrices of linear forms that drop rank along some point on $\pp^1$. 
By excision, we have a right exact sequence
\begin{equation} \label{res}
A^{*-\mathrm{codim}X}(X) \rightarrow A^*(M)  \rightarrow A^*(\V_{r,d}) \rightarrow 0.
\end{equation}
We shall see soon that $X$ is irreducible of codimension $r$, which immediately implies there are no relations among generators of $A^*(M)$ restricted to $\V_{r,d}$ in degrees less than $r$. In what follows, we describe relations among the restrictions of these Chern classes in degrees $r$ up to $d$.

First we construct a space $\widetilde{X}$ whose total space maps properly to $M$ with image $X$.
Let $\pp(T_d)$ be the projectivization of the tautological rank $d$ bundle and let $\sigma: \pp \W \times_B \pp(T_d) \rightarrow B$ be the map to the base. On $\pp \W \times_B \pp(T_d)$ we have a surjection of vector bundles 
\[\sigma^*M = \sigma^*(\mathcal{H}om(T_d, T_{r+d}) \otimes \W^\vee) \rightarrow \O_{\pp(T_d)}(1) \otimes \sigma^*T_{r+d} \otimes \O_{\pp \W}(1),\]
corresponding to evaluation of the map along a one-dimensional subspace of the fiber of $T_d$. Let $\widetilde{X} \subset\sigma^*M$ denote the total space of the kernel vector bundle.
Informally,
\[\widetilde{X} = \{(p, \Lambda, \psi) : p \in \pp \W, \Lambda \subset (T_d)_{\pi(p)},  \psi \in M_{\pi(p)}, \Lambda \subset \ker \psi(p) \subset (T_d)_{\pi(p)}\} .\]

 We have a commutative diagram:
\begin{center}
\begin{tikzcd}
&\widetilde{X} \arrow{dr}[swap]{\rho''} \arrow{r}{\iota} &\sigma^*M\arrow{d}{\rho'} \arrow{r}{\sigma'} &M \arrow{d}{\rho} \\
&&\pp \W \times_B \pp(T_d) \arrow{r}{\sigma} & B,
\end{tikzcd}
\end{center}
where $\rho, \rho',$ and $\rho''$ are all vector bundle maps. 
By construction, $\sigma'(\iota(\widetilde{X})) = X$ and $\sigma' \circ \iota$ is projective, as desired. The map $\widetilde{X} \to X$ is also generically $1$-to-$1$. This establishes that
\[\dim X = \dim \widetilde{X} = \dim M + \dim (\pp \W \times_B \pp(T_d)) - \rank T_{k+r} = \dim M - r.\] 
Since $\sigma' \circ \iota$ is projective, the pushforward of rational Chow groups $A_\qq^*(\widetilde{X}) \to A_\qq^*(X)$ is surjective. (Take the preimage of any cycle; if the fibers are generically finite, the pushforward is some multiple of the original cycle. Otherwise, slice with enough copies of the hyperplane class to get a cycle mapping with generically finite fibers to the same image.) It follows that the image of $A_\qq^{*-r}(X) \rightarrow A_\qq^*(M)$ via pushforward of cycles is the same as the image of $(\sigma' \circ \iota)_*: A_{\qq}^{*-r}(\widetilde{X}) \rightarrow A_{\qq}^*(M)$.

The pullback maps $(\rho'')^*$, $(\rho')^*$ and $\rho^*$ all induce isomorphisms on Chow rings.
Let $(\rho')^*\beta$ be the fundamental class of $\widetilde{X}$ in the Chow ring of $\sigma^*M$. In fact, 
\[\beta := c_{r+d}(\O_{\pp(T_d)}(1) \otimes \sigma^*T_{r+d} \otimes  \O_{\pp \W}(1)).\]
We can write every class in $A^{*-r}(\widetilde{X})$ as $(\rho'')^* \alpha$ for a unique class $\alpha \in A^{*-r}(\pp \W \times_B \pp(T_d))$. 
Then the effect of $(\sigma' \circ \iota)_*$ is
\begin{equation} \label{effect}
\sigma'_* \iota_* (\rho'')^* \alpha = \sigma'_* \iota_* \iota^* (\rho')^* \alpha = \sigma'_*([\widetilde{X}] \cdot (\rho')^*\alpha) = \sigma'_*(\rho')^*(\beta \cdot \alpha) = \rho^*\sigma_*(\beta \cdot \alpha).
\end{equation}
Now, let $z = c_1(\O_{\pp \W}(1))$ and $\zeta = c_1(\O_{\pp(T_d)}(1))$.
By the projective bundle theorem,
\begin{equation} \label{ch}
A^*(\pp \W \times_B \pp(T_d)) = A^*(B)[z,\zeta]/(z^2 +w_1z+w_2, \zeta^{d} + \zeta^{d-1}t_1 + \ldots + \zeta t_{d-1} + t_{d}).
\end{equation}
Thus, each class $\alpha \in A^*(\pp \W \times_B \pp(T_d))$ is uniquely expressible as 
\begin{equation} \label{expressible}
\alpha = \sum_{i=0}^{d-1} (\sigma^*\gamma_i) \zeta^i + \sum_{i=1}^d (\sigma^*\gamma_i') z\zeta^{i-1}, 
\end{equation}
where $\gamma_i \in A^*(B)$ with $\deg \gamma_i = \deg \gamma_i' = \deg \alpha - i$. 

By \eqref{effect} and \eqref{expressible}, the image of $(\sigma' \circ \iota)_*: A_{\qq}^{*-r}(\widetilde{X}) \rightarrow A_{\qq}^*(X)  \cong A_{\qq}^*(B)$ is the ideal generated by 
the classes
\[f_{i,j} := \sigma_*(\beta \cdot z^i \zeta^j) \qquad \text{for }0 \leq i \leq 1,  \ 0 \leq j \leq d - 1.\]
As $\rho^*$ is an isomorphism on Chow, we omit it above and in what follows for ease of notation.
By the excision sequence \eqref{res}, we have
\[A_{\qq}^*(\V_{r,d}) = \frac{A_{\qq}^*(B)}{\langle f_{i,j} : 0 \leq i \leq 1, 0 \leq j \leq d - 1\rangle} \cong \frac{\qq[w_1, w_2, t_1, \ldots, t_d, u_1, \ldots, u_{r+d}]}{\langle f_{i,j} : 0 \leq i \leq 1, 0 \leq j \leq d - 1\rangle}.\]

Since $\sigma$ has relative dimension $d$, the codimension of $f_{i,j}$ is $(r+d)+i+j - d = r+i+j$.
Recall that there are no relations among the generators of $A^*(M) \cong A^*(B)$, so $f_{i,j}$ is a unique polynomial of codimension $i+j+r$ in the $t$'s, $u$'s and $w$'s.
We are interested in particular in the coefficients of $t_{i+j+r}$ and $u_{i+j+r}$ in this expression for $f_{i,j}$.
By the splitting principle and \eqref{ch}, we have
\begin{align*}
\beta &= c_{r+d}(\O_{\pp(T_d)}(1) \otimes \sigma^*T_{r+d} \otimes \O_{\pp \W}(1)) = \sum_{i=0}^{r+d} (\zeta + z)^{r+d-i} \sigma^*u_i \\
&=(\zeta^{r+d} + (r+d)z\zeta^{r+d-1}) + (\zeta^{r+d-1}+(r+d-1)z\zeta^{r+d-2})\sigma^*u_1 \\
& \qquad + \ldots + (\zeta^d+dz\zeta^{d-1})\sigma^*u_r + \ldots + \sigma^*u_{r+d} + \langle \sigma^*w_1, \sigma^*w_2\rangle.
\end{align*}
The push forward of any term involving $\sigma^* w_1$ or $\sigma^*w_2$ cannot contribute to the coefficient of $t_{i+j+r}$ or $u_{i+j+r}$. Since $z^2  \in  \langle \sigma^*w_1, \sigma^*w_2\rangle$, after we multiply $z^i\zeta^j$ with $\beta$, we only care about the resulting terms where the power of $z$ is $1$ (if the power of $z$ is zero, then the push forward vanishes).

To compute the push forward of such terms, iterated use of \eqref{ch} tell us (or c.f. \cite[Corollary 2.6]{HT})
\[\sigma_*(z\zeta^{d-1+i}) = \begin{cases} 0& \text{if $i < 0$} \\ 1 &\text{if $i = 0$} \\ \sum_{m_1 \cdot 1 + \ldots + m_d \cdot d = i} (-1)^{m_1 + \ldots + m_d} \cdot \frac{(m_1 + \ldots + m_d)!}{m_1! \cdots m_d!} \cdot t_1^{m_1} \cdots t_d^{m_d} &\text{if $i \geq 1$.}
\end{cases}\]
The coefficient in front of a monomial for $(m_1, \ldots, m_d)$ above is the number of ordered partitions of $i$ so that $j$ appears with multiplicity $m_j$. The terms we are interested in will come from that monomial being $1$ or $t_{i}$. In particular, we compute
\begin{align*}
f_{1, j-1} &= \sigma_*(\beta z \zeta^{j-1}) = - t_{j + r} + u_{j + r} + \ldots\\
f_{0, j} &= \sigma_*(\beta \zeta^j) = -(r+d)t_{j+r} + (d - j)u_{j+r} + \ldots.
\end{align*}
Hence, in $A_{\qq}^*(\V_{r,d})$,
the classes $t_{n}$ for  $r \leq n \leq d$ and $u_{m}$ for $r+1 \leq m \leq d$ are expressible as polynomials in $w_1, w_2, t_1, \ldots, t_{r-1}, u_1, \ldots, u_r$. Moreover, after eliminating these higher degree generators, the $f_{i,j}$ produce no additional relations in degrees less than or equal to $d$ among the restrictions to $\V_{r,d}$ of $w_1, w_2, t_1, \ldots, t_{r-1}, u_1, \ldots, u_r$. Hence, the map
\begin{equation} \label{gen}
\qq[w_1, w_2, t_1,t_2,\ldots, t_{r-1}, u_1, \ldots, u_{r}] \to A_{\qq}^*(\V_{r,d})
\end{equation}
is an isomorphism in degrees $* \leq d$.
For dimension reasons, there can be no relations among the generators $w_1, w_2, a_1, \ldots, a_r, a_2', \ldots, a_r'$ in degrees less than $d$ because
they are a list of the same number of generators in the same degrees as \eqref{gen}.
\end{proof}

\begin{Corollary} \label{ratch}
We have $A_{\qq}^*(\B_{r,\ell}) = \qq[w_1, w_2, a_1, \ldots, a_r, a_2', \ldots, a_r']$.
\end{Corollary}
\begin{proof}
Recall that the codimension of the complement of $\U_{m, r, \ell}$ is $\ell + mr + 1$. Thus, choosing $m$ such that $\ell + mr + 1 > i$, we see 
\begin{align*}
\dim A_{\qq}^i(\B_{r,\ell}) &= \dim A_{\qq}^i(\U_{m, r, \ell}) = \dim A_{\qq}^i(\V_{r,mr+\ell}) \\
&= \dim \qq[w_1, w_2, a_1, \ldots, a_r, a_2', \ldots, a_r']_i.
\end{align*}
We already know $ \qq[w_1, w_2, a_1, \ldots, a_r, a_2', \ldots, a_r'] \twoheadrightarrow A^*(\B_{r,\ell})$, so we conclude it has no kernel for dimension reasons.
\end{proof}

\section{The integral Chow ring} \label{int-sec}
Our plan is to describe the integral Chow ring of $\B_{r,\ell}$ by building it up as a union of the splitting loci $\Sigma_{\vec{e}}$ using excision. 
The reader may be interested to compare the ideas and use of higher Chow groups here with work of Bae and Schmidtt in their study of the moduli stack of pointed genus $0$ curves \cite{BS}.

The Chow rings of our strata $\Sigma_{\vec{e}}$ are particularly nice. 
Fix some $\vec{e}$ and write $\O(\vec{e}) = \bigoplus_{i=1}^s \O(d_i)^{\oplus r_i}$ as in Section \ref{sps}.
There is an inclusion of groups $\prod \GL_{r_i} \times \GL_2 \hookrightarrow H_{\vec{e}}$, which induces a map on classifying stacks
\begin{equation} \label{vb}
\prod \BGL_{r_i} \times \BGL_2 \rightarrow BH_{\vec{e}}.
\end{equation}

\begin{Lemma} \label{affbun}
The map \eqref{vb} factors as a sequence of affine bundles.
In particular $A^*(\Sigma_{\vec{e}}) =A^*(\prod \BGL_{r_i} \times \BGL_2)$, which is a free $\zz$ algebra. Furthermore, over $\cc$, the first higher Chow groups satisfy
\[\CH^{*}(\Sigma_{\vec{e}},1, \zz/p\zz) = \CH^{*}\left(\prod \BGL_{r_i} \times \BGL_2, 1,\zz/p\zz\right)  = 0\]
 for all primes $p$.
\end{Lemma}

\begin{proof}
We induct on $s$. The $s=1$ case is immediate. Let $G' = \Aut(\oplus_{i=1}^{s-1} \O(d_i)^{\oplus r_i})$ and $H = G' \times GL_{r_s} \subset \Aut(\O(\vec{e}))$. There is a quotient map $\Aut(\O(\vec{e})) \rightarrow H$ defined by forgetting blocks not in $H$, which expresses $\Aut(\O(\vec{e}))$ as a semidirect product $N \rtimes H$ where $N \cong H^0(\O_{\pp^1}(d_s - d_1))^{\oplus (r_1r_s)} \oplus \cdots \oplus H^0(\O_{\pp^1}(d_s - d_{s-1}))^{\oplus (r_{s-1}r_s)}$ is affine. Moreover, 
\[H_{\vec{e}} = (N \rtimes H) \rtimes \GL_2 = N \rtimes ((G' \rtimes \GL_2) \times \GL_{r_s}).\]
Now let $H_{\vec{e}}$ act on $N$ where elements of $N$ act by left multiplication and elements of $(G' \rtimes \GL_2) \times \GL_{r_s}$ act by conjugation. 
This action is affine linear, so the quotient, which is $B((G' \rtimes \GL_2) \times \GL_{r_s})$, is an affine bundle over $BH_{\vec{e}}$.
By the homotopy property, the Chow ring and higher Chow groups of $BH_{\vec{e}}$ agree with that of $B(G' \rtimes \GL_2) \times \BGL_{r_s}$, which, by induction, are isomorphic to those of $(\BGL_{r_1} \times \cdots \times \BGL_{r_{s-1}} \times \BGL_2) \times \BGL_{r_s}$. The vanishing of the first higher Chow group over $\cc$ is equation \eqref{hch}.
\end{proof}

Using the above lemma, we find that inclusion of cycle classes from strata is injective and deduce that the Chow ring  of $\B_{r,\ell}$ is torsion-free.

\begin{Lemma} \label{5}
The Chow group $A^i(\B_{r,\ell})$ is a finitely-generated free $\zz$-module for all $i$.
\end{Lemma}
\begin{proof}
Suppose $U$ is a finite union of strata and $\Sigma_{\vec{e}}$ is a disjoint stratum which is closed in $X = U \cup \Sigma_{\vec{e}}$.
By Lemma \ref{affbun}, we know $A^{i}(\Sigma_{\vec{e}})$ is a finitely-generated free $\zz$-module for all $i$.  By induction, we may also assume $A^i(U)$ is a finitely-generated free $\zz$-module. 
It suffices to show that $A^i(X)$ is also a finitely-generated free $\zz$-module.

Let us first deduce the result over $\cc$.
We have a localization long exact sequence
 \begin{align*}
 \ldots &\rightarrow \CH^{*-u(\vec{e})}(\Sigma_{\vec{e}}, 1, \zz/p) \rightarrow \CH^*(X,1,\zz/p) \rightarrow \CH^*(U,1,\zz/p) \\
 &\rightarrow A^{*-u(\vec{e})}(\Sigma_{\vec{e}}) \otimes \zz/p\zz \rightarrow A^*(X) \otimes \zz/p\zz \rightarrow A^*(U) \otimes \zz/p\zz \rightarrow 0.
 \end{align*}
By Lemma \ref{affbun}, we have $\CH^{*-u(\vec{e})}(\Sigma_{\vec{e}}, 1, \zz/p)=0$, so if
 $\CH^{*}(U, 1, \zz/p\zz) = 0$ then we have $\CH^*(X,1,\zz/p) = 0$ too. It follows by induction that $\CH^*(U,1,\zz/p) = 0$ for any finite union of strata.
Hence, we have an exact sequence
 \begin{equation} \label{forallp}
 0 \rightarrow A^{*-u(\vec{e})}(\Sigma_{\vec{e}}) \otimes \zz/p\zz \rightarrow A^*(X) \otimes \zz/p\zz \rightarrow A^*(U) \otimes \zz/p\zz \rightarrow 0
 \end{equation}
 for each prime $p$.
Because the $\zz$-modules involved are finitely-generated, the exactness of \eqref{forallp} for all $p$ implies that
\begin{equation} \label{aseq}
0 \rightarrow A^{i-u(\vec{e})}(\Sigma_{\vec{e}}) \rightarrow A^i(X)  \rightarrow A^i(U) \rightarrow 0
\end{equation}
is exact.
Since $A^{i-u(\vec{e})}(\Sigma_{\vec{e}})$ and $A^i(U)$ are finitely-generated free $\zz$-modules, so is $A^i(X)$. This concludes the proof over $\cc$.

The exactness of \eqref{aseq} over $\cc$ also tells us that
\begin{equation} \label{rkeq}
\rank A^i(\B_{r,\ell}) = \sum_{\vec{e}} \rank A^{i-u(\vec{e})}(\Sigma_{\vec{e}}).
\end{equation}
We claim that \eqref{rkeq} in fact holds over any ground field. Indeed, Theorem \ref{ratchow} holds over any field, so the
left-hand side is independent of the ground field. Similarly, using Lemma \ref{affbun}, we have $A^*(\Sigma_{\vec{e}}) = A^*(\prod \BGL_{r_i} \times \BGL_2)$ and the latter is independent of the ground field, so the right-hand side of \eqref{rkeq} is independent of the ground field.

Now, working over any ground field, we claim that the map $A^{i-u(\vec{e})}(\Sigma_{\vec{e}}) \to A^i(X)$ for attaching each stratum is injective. If not, then the image of $A^{i-u(\vec{e})}(\Sigma_{\vec{e}}) \to A^i(X)$ would have rank strictly less than $\rank A^{i - u(\vec{e})}(\Sigma_{\vec{e}})$. Then, $\rank A^i(\B_{r,\ell})$ would be less than the sum $\sum_{\vec{e}} \rank A^{i-u(\vec{e})}(\Sigma_{\vec{e}})$, violating \eqref{rkeq}. Hence, we must have an exact sequence as in \eqref{aseq} for each $\vec{e}$. Arguing as before, we know $A^{i-u(\vec{e})}(\Sigma_{\vec{e}})$ and $A^i(U)$ are finitely-generated free $\zz$-modules, so because \eqref{aseq} is exact, $A^i(X)$ is also a finitely-generated free $\zz$ module.
 \end{proof}
 
An analogous argument in cohomology can be used to show that Chow and cohomology rings of $\B_{r,\ell}$ agree.
 
 \begin{Lemma} \label{coh}
 The cycle class map $A^*(\B_{r,\ell}) \to H^{2*}(\B_{r,\ell})$ is an isomorphism.
 \end{Lemma}
\begin{proof}
As before, suppose $U$ is a finite union of strata and $\Sigma_{\vec{e}}$ is a disjoint stratum which is closed in $X = U \cup \Sigma_{\vec{e}}$. 
Because $\Sigma_{\vec{e}}$ and $X$ are smooth, the cohomology of the pair $H^*(X, U)$ is the reduced cohomology of the Thom space of the normal bundle of $\Sigma_{\vec{e}} \subset X$. By the Thom isomorphism, the cohomology of this pair is then $H^*(X, U) \cong H^{* - 2u(\vec{e})}(\Sigma_{\vec{e}})$.
By Lemma \ref{affbun}, we have $H^*(\Sigma_{\vec{e}}) = H^*(\prod \BGL_{r_i} \times \BGL_2)$, which vanishes in odd degrees.
By induction, we may assume that the odd cohomology of $U$ vanishes. Thus, the long exact sequence for the pair $(X, U)$ gives an exact sequence
\begin{equation} \label{cohseq}
0 \rightarrow H^{2(i - u(\vec{e}))}(\Sigma) \rightarrow H^{2i}(X) \rightarrow H^{2i}(U) \rightarrow 0
\end{equation}
for each $i$.
The cycle class map sends the short exact sequence \eqref{aseq} to \eqref{cohseq}. Because $A^*(\BGL_{r_i}) \to H^{2i}(\BGL_{r_i})$ is an isomorphism, 
using Lemma \ref{affbun}, we see that
$A^{i - u(\vec{e})}(\Sigma_{\vec{e}}) \to H^{2(i - u(\vec{e}))}(\Sigma_{\vec{e}})$ is an isomorphism. By induction, we may assume $A^i(U) \to H^{2i}(U)$ is an isomorphism. By the $5$-lemma, we see that $A^i(X) \to H^{2i}(X)$ is also an isomorphism.
\end{proof}

\begin{proof}[Proof of Theorem \ref{main}]
Corollary \ref{ratch} determines $A_{\qq}^*(\B_{r,\ell})$. Lemma \ref{5} shows $A^*(\B_{r,\ell})$ is a subring of $A_{\qq}^*(\B_{r,\ell})$ and Lemma \ref{int-gen} identifies it as
the subring generated by the Chern classes of the sheaves $\pi_*\E(m)$. Lemma \ref{coh} shows that the cycle class map is an isomorphism.
\end{proof}

To make this subring more explicit, we provide formulas for the Chern classes of $\pi_* \E(m)$ in terms of the rational generators (working modulo $\langle w_1, w_2 \rangle$). Recall that $a_1' = \ell$ is the relative degree of $\E$.
\begin{Lemma} \label{F} Let
\begin{equation} \label{Feq}
F(t) = \sum_{i=0}^\infty f_i t^i = \exp\left(\int \frac{a_1'(a_1 +a_2t +\ldots + a_r t^{r-1}) - (a_2' + a_3't + \ldots + a_r' t^{r-2})}{1 + a_1t + \ldots + a_rt^r}dt \right).
\end{equation}
Then
\begin{equation}\label{cherngen}
\sum_{i=0}^\infty c_i(\pi_*\E(m)) t^i = F(t)(1 + a_1t + \ldots + a_rt^r)^{m+1} \quad \mod \langle w_1, w_2, t^{mr+\ell+1} \rangle.
\end{equation}
\end{Lemma}

\begin{proof}
We will use Grothendieck-Riemann-Roch to compute the Chern characters and make use of formal manipulations that turn power sums (Chern characters) into elementary symmetric functions (Chern classes). It is convenient to package this information in generating functions.

Given some $\alpha_1, \ldots, \alpha_r$, let $\sigma_j := \sigma_j(\alpha_1, \ldots, a_r) = \sum_{i_1 < i_2 < \cdots < i_j} \alpha_{i_1}\alpha_{i_2} \cdots a_{i_j}$ denote the $j$th elementary symmetric function in the $\alpha_i$. For each $j \geq 0$, there is a polynomial $p_j(x_1, \ldots, x_r)$ such that $p_j(\sigma_1, \ldots, \sigma_r) = \alpha_1^j + \ldots + \alpha_r^j$. These polynomials satisfy
\begin{equation} \label{identity}
 \log (1+x_1t +\ldots + x_rt^r) = \sum_{j=1}^\infty \frac{(-1)^{j+1}}{j} p_j(x_1,\ldots,x_r) t^j.
\end{equation}
For any vector bundle $E$, the $j$th Chern character is related to the Chern classes by
\[ \ch_j(E) = \frac{1}{j!} p_j(c_1(E), \ldots, c_r(E)).\]
Each Chern character of the universal $\E$ on $\pi: \pp \W \to \B_{r,\ell}$ is expressible as $\ch_j(\E) = \pi^*(c_j) + \pi^*(c_j')z $ for some $c_j \in A_\qq^{j}(\B_{r,\ell})$ and $c_j' \in A_\qq^{j-1}(\B_{r,\ell})$. 
In what follows, we work modulo the ideal $\langle w_1, w_2 \rangle$, so that $z^2 = 0$. We have
\begin{align} \label{c}
c_j = \frac{1}{j!} p_j(a_1, \ldots, a_r) \qquad \text{and} \qquad c_j' = \frac{1}{j!} \sum_{i=1}^r a_i' \frac{\partial p_j}{\partial x_i}(a_1, \ldots, a_r).
\end{align}
We write $c = c_0 + c_1 + c_2 + \ldots$ and $c' = c_1' + c_2' + \ldots$, so $\ch(\E(m)) = \pi^*(c) + \pi^*(c')z$.

The relative tangent bundle of $\pi: \pp \W \rightarrow \B_{r,\ell}$ has Todd class 
$\td(T_\pi) = 1 + \frac{1}{2} c_1(T_\pi) = 1+z.$
Moreover, $\ch(\E(m)) = \ch(\E)\ch(\O_{\pp \W}(m)) = \ch(\E)(1+mz)$.
 On $\U_m$, we have $R^1\pi_* \E(m) =0$ so Grothendieck-Riemann-Roch tells us
\begin{align*}
\ch(\pi_*\E(m)) &= \pi_*(\ch(\E(m))\td(T_\pi)) \\
&= \pi_*((\pi^*(c)+\pi^*(c')z)(1+(m+1)z)) = c' + (m+1)c.
\end{align*}
To recover the Chern classes, we evaluate
{\small
\begin{align}
&\exp\left(\sum_{j=1}^\infty (j-1)! (-1)^{j+1}\ch_j(\pi_*\E(m)) t^j\right) = \exp\left(\sum_{j=1}^\infty (j-1)! (-1)^{j+1}(c_{j+1}' + (m+1)c_j )t^j\right)  \notag \\
&\qquad \qquad =\exp\left(\sum_{i=1}^r a_i'\sum_{j=1}^\infty \frac{(-1)^{j+1}}{j(j+1)} \frac{\partial p_{j+1}}{\partial x_i}(a_1,\ldots,a_r)t^j\right) (1+a_1t+\ldots+a_rt^r)^{m+1} \label{exp}.
\end{align}}
To evaluate the infinite sums inside \eqref{exp}, we consider
 \begin{align}
&\frac{d}{d t} \left(\sum_{j=1}^\infty \frac{(-1)^{j+1}}{j(j+1)} \frac{\partial p_{j+1}}{\partial x_i}(a_1, \ldots, a_r) t^j\right) = \sum_{j=1}^\infty \frac{(-1)^{j+1}}{j+1} \frac{\partial p_{j+1}}{\partial x_i}(a_1, \ldots, a_r) t^{j-1}  \notag \\
&\qquad\qquad = \sum_{j=2}^\infty \frac{(-1)^j}{j} \frac{\partial p_j}{\partial x_i}(a_1, \ldots, a_r)t^{j-2}. \notag
\intertext{Note that $\partial p_j/\partial x_i = 0$ if $j < i$. Therefore taking the partial derivative of \eqref{identity} with respect to $x_i$, we see that the above is equal to} 
&\qquad\qquad = \frac{-1}{t^2} \left(\frac{\partial}{\partial x_i} \log(1 + a_1 t + \ldots + a_r t^r) - \delta_{i1} t\right) \notag \\
&\qquad\qquad = \begin{cases} (a_1 + a_2 t + \ldots + a_rt^{r-1})/(1 + a_1 t + \ldots + a_r t^r) & \text{if $i = 1$} \\ -t^{i-2}/(1 + a_1t + \ldots + a_rt^r) & \text{if $i \geq 2$.} \end{cases} \label{rats}
\end{align}
Substituting the formal integral of the terms in \eqref{rats} into \eqref{exp} and exponentiating \eqref{exp} then gives the formula \eqref{cherngen}.
\end{proof}

\begin{remark}
The Chern classes of $\pi_*\E(m)$ (not just modulo $\langle w_1, w_2 \rangle$) are also expressible in terms of exponentials of formal integrals of rational functions in the $a_i$ and $a_i'$ and the Chern classes of $\W$. Equation \eqref{c} is made up of terms with higher partial derivatives of the polynomials $p_j$ multiplied by the Chern classes of $\W$. The sums of these higher partial derivatives that appear in expanding \eqref{exp} can again be collected into rational functions and formal integrals of rational functions in the $a_i$.
However, we no longer have the sequence $0 \rightarrow \pi_*\E(m-1) \rightarrow \pi_*\E(m) \rightarrow \E|_{B \times \{0\}} \rightarrow 0$ that holds for vector bundles on trivial families $B \times \pp^1$ so the relationship between twists is not as nice.
\end{remark}

\begin{proof}[Proof of Theorem \ref{m2}]
By Vistoli's theorem, we have $A^*(\B_{r,\ell}^\dagger) = A^*(\B_{r,\ell})/\langle w_1, w_2 \rangle$. By Lemma \ref{F}, after setting $w_1 = w_2 = 0$, the Chern classes of $\pi_*\E(m)$ for all $m$ are expressible in terms the $f_i$ in \eqref{Feq} and $a_1, \ldots, a_r$. To see that the cycle class map is an isomorphism, one may argue as in Lemmas \ref{affbun} and \ref{coh}, but using the splitting loci $\Sigma^{\dagger}_{\vec{e}}$ on $\B_{r,\ell}^\dagger$. One has $\Sigma_{\vec{e}}^{\dagger} \cong B\Aut(\O(\vec{e}))$, so the proof is essentially the same after dropping the $\BGL_2$ factor from Lemma \ref{affbun}.
\end{proof}

\begin{Corollary} \label{nfg}
The integral Chow rings $A^*(\B_{r,\ell})$ and $A^*(\B_{r,\ell}^\dagger)$ are not finitely generated as $\zz$ algebras.
\end{Corollary}
\begin{proof}
Since $A^*(\B_{r,\ell})$ surjects onto $A^*(\B_{r,\ell}^\dagger)$, it suffices to prove the claim for $\B_{r,\ell}^\dagger$.
Suppose to the contrary that $A^*(\B_{r,\ell}^\dagger)$ were finitely generated as a $\zz$-algebra. Then there would exist a finite set of primes $\{p_1, \ldots, p_n\}$ such that for all $\eta \in A^*(\B_{r,\ell}^\dagger)$, there exist $m_1, \ldots, m_n \in \zz$ such that $p_1^{m_1} \cdots p_n^{m_n} \cdot \eta \in \zz[a_1, \ldots, a_r, a_2', \ldots, a_r']$. Let $q$ be a prime not in $\{p_1, \ldots, p_n\}$. Let us consider the coefficient $f_q$ in \eqref{Feq}. Performing the formal integral, we see that $F(t) = \exp(-a_2' t + \ldots)$, so $(a_2')^q$ appears in $f_q$ with coefficient $\frac{1}{q!}$. Hence, $p_1^{m_1} \cdots p_n^{m_n} \cdot f_q \notin \zz[a_1, \ldots, a_r, a_2', \ldots, a_r']$ for any $m_1, \ldots, m_n$.
\end{proof}

As an example of the utility of the formulas in Lemma \ref{Feq}, we show that $A^*(\B_{2,1}^\dagger)$ is not isomorphic to $A^*(\B_{2,0}^\dagger)$. A similar argument shows $A^*(\B_{2,1})$ is not isomorphic to $A^*(\B_{2,0})$.

\begin{Corollary} \label{b2}
(1) The integral Chow rings  $A^*(\B_{2,1}^\dagger)$ and  $A^*(\B_{2,0}^\dagger)$ are not isomorphic. Hence, $\B_{2,0}^\dagger$ is not isomorphic to $\B_{2,1}^\dagger$. 

(2) The integral Chow rings  $A^*(\B_{2,1})$ and  $A^*(\B_{2,0})$ are not isomorphic. Hence, $\B_{2,0}$ is not isomorphic to $\B_{2,1}$.
\end{Corollary}
\begin{proof}
(1) Let $\sum f_i t^i = F(t)$ as in \eqref{Feq}. 
Then $A^*(\B_{2,\ell}^\dagger) \subset \qq[a_1, a_2, a_2']$ is the subring generated by the classes $f_i$ together with $a_1, a_2, a_2'$. We must show that we obtain non-isomorphic rings (not just different subrings) when $\ell = a_1' = 0$ versus when $\ell = a_1' = 1$.

Setting $\ell = a_1' = 0$, we see $A^2(\B_{2,0}^\dagger)$ is generated by
\begin{align*}
f_2 = \frac{1}{2}(a_1a_2' + a_2'^2), \quad a_1^2, \quad a_1a_2', \quad a_2 \qquad \qquad &(\ell = 0). \\
\intertext{On the other hand, setting $\ell = a_1' = 1$, we see that
$A^2(\B_{2,1}^\dagger)$ is generated by}
f_2 = \frac{1}{2}(-a_1a_2' + a_2'^2 + a_2), \quad a_1^2, \quad a_1a_2', \quad a_2 \qquad \qquad &(\ell =1).
\end{align*}
Now already, we can see that we have produced different subrings of $\qq[a_1, a_2, a_2']$. To see that our rings are not isomorphic though, we look at classes in codimension $4$.

Let us consider the group $A^4(\B_{2,\ell}^\dagger)/(A^1(\B_{2,\ell}^\dagger) \cdot A^3(\B_{2,\ell}^\dagger) + A^2(\B_{2,\ell}^\dagger)^2)$. By Theorem \ref{m2}, this group is generated by $f_4$. When $\ell = 0$, we have
\[f_4 = \frac{1}{24} \cdot a_2' \cdot (6a_1^3 + 11a_1^2a_2' + 6a_1a_2'^2 + a_2'^3 - 12a_1a_2 - 8a_2a_2') \in A^4(\B_{2,0}^\dagger). \]
In particular, $24f_4 \in A^1(\B_{2,0}) \cdot A^3(\B_{2,0})$. On the other hand, when $\ell = 1$, we claim that no element in the coset $f_4 + (A^1(\B_{2,\ell}) \cdot A^3(\B_{2,\ell}) + A^2(\B_{2,\ell})^2)$ has the property that a multiple of it lies in $A^1(\B_{2,\ell}^\dagger) \cdot A^3(\B_{2,\ell}^\dagger)$. To see this, we compute
\begin{align*}
f_4 &= \frac{1}{24}(-2a_1^3a_2' - a_1^2 a_2'^2 + 2a_1a_2'^3 + a_2'^4 + 2a_1^2a_2 + 6a_1a_2a_2' - 2 a_2 a_2'^2-3a_2^2) \in A^4(\B_{2,1}^\dagger) \\
&= -\frac{1}{8} a_2^2 + \ldots.
\end{align*}
But there is no integral class in 
$A^1(\B_{2,\ell}^\dagger) \cdot A^3(\B_{2,\ell}^\dagger) + A^2(\B_{2,\ell}^\dagger)^2$ that involves $\frac{1}{8} a_2^2$. Thus, there is no adjustment of $f_4$ by integral classes from lower codimension so that the result is divisible by a codimension $1$ class.

(2) The strategy here similar (the formulas are just more complicated because we must also keep track of $w_1$ and $w_2$). The codimension of the complement of $\U_{3, 2, \ell} \subset \B_{2,\ell}$ is $6 + \ell$. In particular, for $\ell= 0, 1$, we have $A^4(\B_{2,\ell}) = A^4(\U_{3,2,\ell})$. 
By Lemma \ref{genslem}, we have that $A^4(\B_{2,\ell})/(A^1(\B_{2,\ell}) \cdot A^3(\B_{2,\ell}) + A^2(\B_{2,\ell})^2)$ is generated by $t_4 := c_4(\pi_*\E(2))$ and $u_4 := c_4(\pi_*\E(3))$. These classes are determined by the splitting principle and Grothendieck--Riemann--Roch, and one can calculate them quickly using the Schubert2 package in Macaulay2. One then observes that if $\ell = 0$, then $t_4$ and $u_4$ are expressible as sum of a class in 
\[\zz[a_1, a_2', a_2, w_1,w_2]_4 \subset A^1(\B_{2,\ell}) \cdot A^3(\B_{2,\ell}) + A^2(\B_{2,\ell})^2,\]
plus a class divisible by $a_2'$. That is, when $\ell = 0$, we can choose generators of 
\[A^4(\B_{2,\ell})/(A^1(\B_{2,\ell}) \cdot A^3(\B_{2,\ell}) + A^2(\B_{2,\ell})^2)\]
 so that a multiple lies in $A^1(\B_{2,\ell})  \cdot A^3(\B_{2,\ell})$. On the other hand, if $\ell = 1$, then $t_4$ and $u_4$ contain an $a_2^2$ term with denominator $8$. However, all codimension $2$ classes have denominators at most $2$. Thus, there is no adjustment of $t_4$ or $u_4$ by integral classes from lower codimension so that the result is divisible by a codimension $1$ class.
\end{proof}

\bibliographystyle{amsplain}
\bibliography{refs}

\end{document}